\documentclass[12pt, oneside,reqno]{amsart}
\usepackage{hyperref}
\usepackage{geometry}
\usepackage{amsmath}
\usepackage{cite}
\usepackage{xcolor}
\usepackage{graphicx}
\usepackage{subfigure}

\usepackage{amsmath}
\allowdisplaybreaks

\newcommand{\rn}{\mathbb R^n}

\newcommand\wtilde[1]{\overset{\lower.4ex\hbox{$\scriptstyle \sim$}}{#1}}

\numberwithin{equation}{section}

\newtheorem{thm}{Theorem}[section]

\newtheorem{lem}[thm]{Lemma}

\newtheorem*{remark*}{Remark}

\newtheorem{rem}[thm]{Remark}

\begin{document}
	\title{$L_p$ Brunn-Minkowski inequality for weighted dual quermassintegrals}

	\author[X. Chen]{Xiaojuan Chen}
	\address{Institute of Mathematics,
		Hunan University,  Changsha,  410082,  China}
	\email{
		cxj@hnu.edu.cn}
	
	\author[S. Tang]{Shengyu Tang}
	\address{Institute of Mathematics,
		Hunan University,  Changsha,  410082,  China}
	\email{
		tsy@hnu.edu.cn}
	
	\author[S. Wang]{Sinan Wang}
	\address{Institute of Mathematics,
		Hunan University,  Changsha,  410082,  China}
	\email{
		wangsinan@hnu.edu.cn}


 \thanks{Research of Chen was supported by the Postdoctoral Fellowship Program of CPSF under Grant Number GZB20250702.}

	\begin{abstract}
We investigate the $L_p$ Brunn-Minkowski inequality for dual quermassintegrals in weighted measure spaces, which is a special class of rotationally invariant measures proposed by Cordero-Erausquin and Rotem [Ann. Probab., {\bf 51} (2023)]. Specifically, the weighted dual quermassintegral is defined by integrating the radial density $|x|^{q-n}\phi(|x|)$ for $q\in(0,n]$, where $\phi$ is a positive radially non-increasing weight, it recovers the classical dual quermassintegral when $\phi\equiv1$.
For $p\geq1$, we prove the $L_p$ Brunn-Minkowski inequality with concavity exponent $p/q$
 under the condition that 
$t\mapsto\log\phi(e^t)$ is concave, which is exactly the natural convexity condition from Cordero-Erausquin and Rotem's paper in general, improving the exponent $1/n$ when $p=1$. For $p\in(0,1)$, we obtain the result with exponent $p/q$ under more strictly weight assumptions, together with explicit lower bounds for the admissible range of $p$.


	\end{abstract}
	\maketitle	
    \section{introduction}
    As the cornerstone of Brunn-Minkowski theory, Brunn-Minkowski inequality is a sophisticated and powerful tool which can be used to solve many kinds of problems involving metric quantities such as
volume, surface area and mean width. 
The classical Brunn-Minkowski inequality states that for any compact sets $K,L\subseteq \mathbb{R}^n$,
\begin{equation}\label{1}
 V_n((1-\lambda)K+\lambda L)^{\frac{1}{n}}\geq(1-\lambda)V_n(K)^{\frac{1}{n}}+\lambda V_n(L)^{\frac{1}{n}},\quad \forall~\lambda\in [0,1]   
\end{equation}
with equality if and only if $K$ and $L$ are homothetic, where $V_n$ is the $n$-dimensional Lebesgue measure (volume). Due to the homogeneity of the volume, the dimensional Brunn-Minkowski inequality \eqref{1} is equivalent to the following dimension free Brunn-Minkowski inequality
\begin{equation}\label{2}
 V_n((1-\lambda)K+\lambda L)\geq V_n(K)^{1-\lambda}V_n(L)^{1-\lambda},\quad \forall~\lambda\in [0,1].   
\end{equation}
And the above measure can be generalized to $s$-concave measures. We say a measure $\mu$ in $\mathbb{R}^n$ is $s$-concave if for any compact sets $K,L\subseteq \mathbb{R}^n$ with $\mu(K),\mu(L)>0$,
\begin{equation}\label{3}
\mu((1-\lambda)K+\lambda L)^s\geq(1-\lambda)\mu(K)^s+\lambda \mu(L)^s,\quad \forall~\lambda\in [0,1]
\end{equation}
with equality if and only if $K$ and $L$ are homothetic. More details for equality characterization can refer to \cite{MR14}. Similarly, a measure $\mu$ is called log-concave (i.e., $0$-concave) if for any  
compact sets $K,L\subseteq \mathbb{R}^n$,
\begin{equation}\label{4}
\mu((1-\lambda)K+\lambda L)\geq \mu(K)^{1-\lambda}\mu(L)^{1-\lambda},\quad \forall~\lambda\in [0,1].  
\end{equation}
Borell \cite{B75} proved that a measure $\mu$, with non-degenerate support, is log-concave if and only if $\mu$ has the density $\varphi$, where $\varphi$ is a log-concave function in $\mathbb{R}^n$, i.e. $\log\varphi$ is concave. 
Notably, inequality \eqref{2} states that the Lebesgue measure is log-concave, while \eqref{1} means it is also $\frac{1}{n}$-concave. But log-concavity will not imply $s$-concavity for $s>0$ without homogeneity of $\mu$. A typical example is the standard Gaussian measure $\gamma_n$, i.e., the measure with density $(2\pi)^{-\frac{n}{2}}e^{-\frac{|x|^2}{2}}$, whose density function is $\log$-concave. From Borell's result, $\gamma_n$ satisfies \eqref{4}. It is obvious that $\gamma_n$ is non-homogeneous, so \eqref{4} can not deduce \eqref{3} for any $s>0$. 
Hence it is natural to ask that whether \eqref{3} holds ture for Gaussian measure $\gamma_n$. The answer is no, and the counterexample was given by
Gardner-Zvavitch \cite{GZ10}. They further  conjectured that the following
Brunn-Minkowski inequality 
\begin{equation*}
    \gamma_n((1-\lambda)K+\lambda L)^{\frac{1}{n}}\geq(1-\lambda)\gamma_n(K)^{\frac{1}{n}}+\lambda\gamma_n(L)^{\frac{1}{n}},\quad \forall~\lambda\in [0,1]
\end{equation*}
holds when $K,L$ are origin-symmetric convex bodies in $\mathbb{R}^n$. 
While important progress was made by Nayar-Tkocz\cite{NT13} and Kolesnikov-Livshyts \cite{KL21}, this conjecture was affirmatively proved by Eskenazis and Moschidis \cite{EM21}. 
Numerous very recent developments concerning log-concave measures can be referred to  \cite{AL25, CE25, MMRR25, XY26}.
Cordero‑Erausquin and Rotem \cite{CR23} further improved this conjecture for all rotationally invariant measures $\mu\in\mathcal{M}_n$, where $\mathcal{M}_n$ is a class of Borel measures in $\rn$ satisfying
\[
\mathcal{M}_n = 
\left\{
\mu=\int_{K}e^{-\omega(|x|)}dx \;\middle|\;
\begin{aligned}
& \omega: (0,+\infty)\rightarrow(-\infty,+\infty]\ \text{is} \\
& \text{non-decreasing and} ~~t\mapsto \omega(e^t) ~~\text{is convex}
\end{aligned}
\right\}.
\]
It is worth noting that for smooth $\omega$, the condition $t\mapsto \omega(e^t)$ is convex in $\mathcal{M}_n$ is equivalent to $\omega''(t)+\frac{\omega'(t)}{t}\geq0$ for $t>0$, which is weaker than the $\log$-concavity of $\mu$, i.e. $\omega''\geq0$.


When $\mu$ is not the Lebesgue measure, the inequality \eqref{3} is sometimes called weighted Brunn-Minkowski inequality. We refer to the seminal series of excellent works for weighted Brunn-Minkowski inequalities by Fradelizi-Langharst-Madiman-Zvavitch \cite{FLMZ241,FLMZ242}. They systematically developed the weighted Brunn-Minkowski theory, extending fundamental notions including mixed volumes, surface area measures, and core inequalities such as the Minkowski and Fenchel inequalities from the classical convex geometry setting to general measures, notably Gaussian measures. A natural extension is the  $L_p$ version of \eqref{3}, which has the form
\begin{equation}\label{L_p-BMI}
\mu((1-\lambda)K+_p\lambda L)^s\geq(1-\lambda)\mu(K)^s+\lambda \mu(L)^s,\quad \forall~\lambda\in [0,1],    
\end{equation}
for origin-symmetric convex bodies $K,L\subset\rn$. And the related results for Gaussian measures (and beyond) can be found in
\cite{HKL21,MNRMI}. 
Furthermore, within such weighted measure spaces, another fundamental part in Brunn-Minkowski theory, the Minkowski problem, has also been extensively investigated, and detailed discussions can be found in 
\cite{Liv19, HXZ21, Hu25, LT24, KL23, LLT}.

Apart from the above progress on weighted Brunn-Minkowski theory and the Minkowski problem, an important conjecture regarding the Brunn-Minkowski inequality for dual quermassintegrals $\widetilde{V}_q(K):=\frac{q}{n}\int_K|x|^{q-n}dx$, which proposed by Lutwak and discussed extensively in Lutwak-Yang-Zhang's group, states that for convex bodies $K, L\subset\mathbb{R}^n$ and $q>0$, 
\begin{equation}\label{E:12}
    \tilde{V}_q((1-\lambda)K+\lambda L)^{\frac{1}{q}}\geq (1-\lambda)\tilde{V}_q(K)^{\frac{1}{q}}+\lambda\tilde{V}_q(L)^{\frac{1}{q}},\quad \forall~\lambda\in[0,1],
\end{equation}
with equality if and only if $K$ and $L$ are homothetic (see survey paper \cite{HYZ25}). Xi-Zhang \cite{XZ22} proved the case $q\in(0,1]$ geometrically. Subsequently, Sadovsky-Zhang \cite{SZ25}, building on the analytic tools of Kolesnikov-Milman \cite{KM17, KM18}, Kolesnikov-Livshyts \cite{KL21}, and the estimate of Cordero-Erausquin and Rotem \cite{CR23}, confirmed the conjecture for $0<q\leq n$ under symmetry assumptions, though equality characterization remains open.


In this paper, we focus on the $L_p$ Brunn-Minkowski inequality for dual quermassintegrals in some rotationally invariant weighted measure space. The functional $\tilde{V}_{q,\phi}$, which we called the weighted dual quermassintegral, is defined as follows: $$\tilde{V}_{q,\phi}(K):=\int_{K}|x|^{q-n}d\mu(x)=\int_K|x|^{q-n}\phi(|x|)dx,\quad K\subset\mathbb{R}^n$$
for some positive density function $\phi:(0,+\infty)\rightarrow(-\infty,+\infty]$ in $\mathcal{N}^\tau_n$, where
\[
\mathcal{N}_n^\tau = 
\left\{
\phi \in C(\mathbb{R}^n_{+}) \;\middle|\;
\begin{aligned}
& \phi \text{ is non-increasing and}~~t \mapsto \log \phi(e^t)\\
& \text{is strongly }\tau\text{-concave}\ \text{for some}~\tau\geq 0
\end{aligned}
\right\}.
\]
In particular, when $\phi\equiv1\in\mathcal{N}_n^0$, the weighted dual quermassintegral $\tilde{V}_{q,\phi}$ corresponds to $\tilde{V}_{q}$, which can be regarded as a natural generalization of the dual quermassintegral. When $\phi=(2\pi)^{-\frac{n}{2}}e^{-\frac{|x|^2}{2}}\in\mathcal{N}_n^0$, which is the density of Gaussian measure $\gamma_n$, the authors in \cite{FHX25, FLX25} considered the corresponding Minkowski problem for $\tilde{V}_{q,\gamma_n}$.

\begin{remark*}
$\tilde{V}_{q,\phi}$ has the unified form of density function $e^{-\omega(|x|)}$ means that $\omega(t)=(n-q)\log t-\log\phi(t)$. Without loss of generality, we may assume that $\phi$ is $C^2$-smooth by approximation. Fix $q\in(0,n]$ and $\tau\geq0$, the condition $t \mapsto \log \phi(e^t)$ is strongly $\tau$-concave in $\mathcal{N}^{\tau}_n$ means that $\frac{d^2(\log \phi(e^t))}{dt^2} \leq -\tau$, which is equivalent to $\frac{d^2(\omega (e^t))}{dt^2} \geq \tau\geq0$. In a word, for any $\tau\geq0$, if $\phi\in\mathcal{N}^\tau_n$, then $\tilde{V}_{q,\phi}\in\mathcal{M}_n$. In particular, when $\tau=0$, the condition $t \mapsto \log \phi(e^t)$ is $0$-concave (i.e. concave) in $\mathcal{N}^{0}_n$ is actually equivalent to the condition $t\mapsto \omega(e^t)$ is convex in $\mathcal{M}_n$.
\end{remark*}

As previously mentioned, Cordero‑Erausquin and Rotem \cite{CR23} proved \eqref{3} for any measure $\mu\in\mathcal{M}_n$ over origin-symmetric convex bodies with $s=\frac{1}{n}$. Our first contribution is that we proved \eqref{3} for $\tilde{V}_{q,\phi}$ when $\phi\in\mathcal{N}^0_n$, $q\in(0,n]$ and improved the exponent $\frac{1}{n}$ to $\frac{1}{q}$. Actually, we further proved the $L_p$ version \eqref{L_p-BMI} with the exponent $s=\frac{p}{q}$ when $p\geq1, q\in(0,n]$. 

\begin{thm}\label{thm8}
Suppose $p\geq 1, q\in(0,n]$ and $\phi\in\mathcal{N}^0_n$,
then for origin-symmetric convex bodies $K,L\subset\mathbb{R}^n$,
  \begin{equation}\label{BMI2}
 \tilde{V}_{q,\phi}((1-\lambda)K+_p\lambda L)^{\frac{p}{q}}\geq(1-\lambda)\tilde{V}_{q,\phi}(K)^{\frac{p}{q}}+\lambda\tilde{V}_{q,\phi}(L)^{\frac{p}{q}}, \quad \forall~\lambda\in [0,1].   
  \end{equation}
\end{thm}
When $p\in(0,1)$, it will be more complicated to derive the $L_p$ Brunn-Minkowski inequality for $\tilde{V}_{q,\phi}$. Therefore, it is crucial to establish a more precise estimate in Lemma \ref{improve of Cordero Rotem}. In order to deal with the smallest eigenvalues $\lambda_{\min}(M)$ in Lemma \ref{improve of Cordero Rotem}, we carry out a classification on the function $\phi$. The second main result is as follows.
\begin{thm}\label{thm9}
Suppose $p\in(0,1)$, $q\in(0,n]$ and $\tau>0$. Then for origin-symmetric convex bodies $K,L\subset\mathbb{R}^n$, the following inequality 
\begin{equation}\label{BMI1}
 \tilde{V}_{q,\phi}((1-\lambda)K+_p\lambda L)^{\frac{p}{q}}\geq(1-\lambda)\tilde{V}_{q,\phi}(K)^{\frac{p}{q}}+\lambda\tilde{V}_{q,\phi}(L)^{\frac{p}{q}}, \quad \forall~\lambda\in [0,1]
  \end{equation}
  holds under the following two cases:
  \begin{itemize}
  \item $t(\log \phi(t))'-t^2(\log\phi(t))''\geq 2(n-q)$ with $\phi\in\mathcal{N}^{0}_n$, $p\in [1-\frac{2(n-q)}{n^2+n+nq-q},1),$ 
  
  \item $t(\log \phi(t))'-t^2(\log\phi(t))''\leq 2(n-q)$ with $\phi\in\mathcal{N}^{\tau}_n$, $p\in [1-\frac{\tau}{n^2+n+nq-q},1)$.
\end{itemize}
\end{thm}

This paper is organized as follows. In Section \ref{Share facts}, we provide some important share facts of the $L_p$ Brunn-Minkowski inequality to prepare for the proof of Theorem \ref{thm8} and \ref{thm9}. In Section \ref{pf1}, we derive the $L_p$-Brunn-Minkowski inequality \eqref{BMI2} for the case $p\geq1$ and complete the proof of Theorem \ref{thm8}. In Section \ref{pf2}, we further prove inequality \eqref{BMI1} for the case $p\in(0,1)$ and complete the proof of Theorem \ref{thm9}.

    \section{Share facts}\label{Share facts}\label{section2}
In this section, we show some important share facts of the $L_p$ Brunn-Minkowski inequality, which will be used to prove Theorem \ref{thm8} and \ref{thm9} later. 

For convenience, in the following we abbreviate $L_p$ Brunn-Minkowski inequality as $L_p$-BMI and denote $\mu$ as a Borel measure in $\mathbb{R}^n$ with density function $\frac{d\mu}{dx}:=\frac{d\tilde{V}_{q,\phi}}{dx}=e^{-W}=|x|^{q-n}\phi(|x|)$. Our aim is to prove the following global $L_p$-BMI  
\begin{equation}\label{bmi}
 \mu((1-\lambda)K+_p\lambda L)^s\geq(1-\lambda)\mu(K)^s+\lambda\mu(L)^s,    \quad \forall~\lambda\in [0,1]
\end{equation}
for origin-symmetric convex bodies $K,L\subset\mathbb{R}^n$. In order to prove Theorem \ref{thm8} and \ref{thm9}, the key is deducing the following local $L_p$-BMI
 \begin{equation}\label{local BMI}
        \frac{d^2}{d\epsilon^2}\mu(K+_p\epsilon f)^s|_{\epsilon=0}\leq0, \quad \forall~f^p\in C^2_e(S^{n-1}).
    \end{equation}
to the global $L_p$-BMI \eqref{bmi}. Then we only need to deal with \eqref{local BMI}.
  The calculation of the first and second derivatives of $\mu(K+_p\epsilon f)$ at $\epsilon=0$ has been shown in \cite{HKL21,KM22},
    \begin{equation}
        \label{caluculation of 1 order}
        \frac{d}{d\epsilon}\mu(K+_p\epsilon f)|_{\epsilon=0}=\int_{\partial K}\psi d\mu_{\partial K},
    \end{equation}
    and
    \begin{equation}\label{calculation of 2 order}
        \frac{d^2}{d\epsilon^2}\mu(K+_p\epsilon f)|_{\epsilon=0}=\int_{\partial K}H_\mu\psi^2d\mu_{\partial K}-\int_{\partial K}\langle\mathbb{II}^{-1}\nabla \psi,\nabla\psi\rangle d\mu_{\partial K}+(1-p)\int_{\partial K}\frac{\psi^2}{h_K}d\mu_{\partial K},
    \end{equation}
    where $\mathbb{II}$ is the second fundamental form of $\partial K$, $H_{\mu}=tr(\mathbb{II})-\langle\nabla W, \nu\rangle$ is the weighted mean curvature and $\psi=\frac{f^p}{ph_K^{p-1}}$ for $p\neq0$. 
   
 Combining \eqref{caluculation of 1 order}, \eqref{calculation of 2 order}, \eqref{local BMI} can be reduced to  more detailed expression,
    \begin{align}\label{Lp expression}
        \int_{\partial K}H_\mu\psi^2-\langle\mathbb{II}^{-1}\nabla \psi,\nabla \psi\rangle d\mu_{\partial K}\leq(p-1)\int_{\partial K}\frac{\psi^2}{h_K}d\mu_{\partial K}+\frac{1-s}{\mu(K)}\left(\int_{\partial K}\psi d\mu_{\partial K}\right)^2.
    \end{align}
    
It is worth noting that the equivalence between \eqref{bmi} and \eqref{local BMI} is obvious when $p\geq 1$. But for the case $p\in (0,1)$, the semi-group property 
    $$h_{K_t}=(1-t)h_{K}+_pth_{L}, \quad \forall ~t\in[0,1]$$
   will not hold, 
   hence it may be not straightforward to establish the equivalence of the ``local'' and ``global'' $L_p$-BMI (more details see \cite{KM22, HKL21}). When $p\in (0,1)$, Chen-Huang-Li-Liu \cite{CHLL20} proved the equivalence of the $L_p$-BMI to its local form using PDE techniques when $\mu$ is Lebesgue measure. This result revealed that the local inequality offers a genuine pathway to the global conjecture, rather than merely a technical reduction. Independently, Putterman \cite{P21} established the same equivalence from a geometric standpoint. He verified that, when $s=\frac{p}{n}$, the local $L_p$-BMI \eqref{local BMI} of volume for origin-symmetric convex bodies is equivalent to the local $\frac{p}{n}$-concavity for strongly isomorphic polytopes. The local concavity was then extended to the global concavity, and yields the global $L_p$-BMI for general origin-symmetric convex bodies by approximation. Their work was subsequently extended by Hosle-Kolesnikov-Livshyts \cite{HKL21} to log-concave measure. Specifically, they showed that verifying the local $L_p$-BMI leads to the global $L_p$-BMI, which is stated as follows.
   \begin{thm}( \cite[Theorem 3.1]{HKL21})
       Assume local $L_p$-BMI \eqref{Lp expression} holds for $s=\frac{q}{n}$ and $p,q<1$. Then for any origin-symmetric convex bodies $K,L\subset\mathbb{R}^n$, we have
       \begin{equation*}
           \mu((1-\lambda)K+_p\lambda L)^{\frac{q}{n}}\geq (1-\lambda)\mu(K)^{\frac{q}{n}}+\lambda\mu(L)^{\frac{q}{n}}, \quad \forall~\lambda\in[0,1].
       \end{equation*}
   \end{thm}  
Actually, if the inequality (21) in \cite[Lemma 3.5]{HKL21} is replaced by \eqref{Lp expression}, then we can deduce the global $L_p$-BMI \eqref{bmi}
by the local $L_p$-BMI \eqref{Lp expression} according to the techniques in the proof of \cite[Theorem 3.1]{HKL21}. 

Next, we will focus on the proof of the local $L_p$-BMI \eqref{Lp expression}.   
 Let $u$ be the solution of Neumann boundary problem
    \begin{equation*}
        \left\{
        \begin{aligned}
            &\Delta u-\langle\nabla W,\nabla u\rangle=\frac{\int_{\partial K}\psi d\mu_{\partial K}}{\mu(K)}, &&\mbox{in}~K, \\
             &u_{\nu}=\psi, &&\mbox{on}~\partial K.
        \end{aligned}\right.
    \end{equation*}
 Since \eqref{Lp expression} is scale-invariant about $\psi$, we may assume that 
 \begin{equation}\label{Lu}
   Lu:=\Delta u-\langle\nabla W,\nabla u\rangle=\frac{\int_{\partial K}\psi d\mu_{\partial K}}{\mu(K)}=1.
 \end{equation}
 Then the generalized Reilly formula in \cite{KM18} and Cauchy-Schwarz inequality show that
 \begin{align}\label{Reilly formula}
     \int_{K}(Lu)^2&-\|\nabla^2u\|_2^2-\langle\nabla^2W\nabla u,\nabla u\rangle d\mu_K \geq\int_{\partial K}H_\mu \psi^2-\langle\mathbb{II}^{-1}\nabla \psi,\nabla \psi\rangle d\mu_{\partial K}.
 \end{align}
 Applying \eqref{Reilly formula} in \eqref{Lp expression}, we only need to prove
 \begin{align*}
     \int_{K}(Lu)^2-\|\nabla^2u\|_2^2-\langle\nabla^2W\nabla u,\nabla u\rangle d\mu_K \leq(p-1)\int_{\partial K}\frac{\psi^2}{h_K}d\mu_{\partial K}+\frac{1-s}{\mu(K)}\left(\int_{\partial K}\psi d\mu_{\partial K}\right)^2.
 \end{align*}
 By divergence theorem and \eqref{Lu}, the above inequality is equivalent to
 \begin{align}\label{equivalent after Newmann problem}
     \int_{K}\|\nabla^2u\|_2^2+\langle\nabla^2W\nabla u,\nabla u\rangle d\mu_K\geq(1-p)\int_{\partial K}\frac{u_{\nu}^2}{h_K}d\mu_{\partial K}+s\mu(K).
 \end{align}

In order to deal with the left hand side of inequality \eqref{equivalent after Newmann problem}, we need the following modified estimates, which were inspired by Cordero-Erausquin and Rotem \cite{CR23}.
\begin{lem}\label{improve of Cordero Rotem}
    Let $d\mu_K:=|x|^{q-n}\phi(|x|)dx=e^{-W}dx$ with $q\in (0,n]$,
    $\phi\in C^2(\mathbb{R}^n_{+})$ is a radially symmetric non-increasing function. 
    Assume that $K$ is an origin-symmetric convex body in $\mathbb{R}^n$ and $u:K\rightarrow\mathbb{R}$ is a smooth even function such that
    \begin{equation}\label{weight laplace equation}
        \Delta u-\langle\nabla W,\nabla u\rangle=1\quad\text{in}\ K.
    \end{equation}
Then for some positive function $\varepsilon=\varepsilon(x)\in[0,1)$,
\begin{equation}\label{improve}
    \int_{K}(1-\varepsilon)\|\nabla^2u\|^2_2+\langle\nabla^2W\nabla u,\nabla u\rangle d\mu_K\geq \int_{K}\lambda_{\min}(M)|\nabla u|^2+\frac{1-\varepsilon}{q}d\mu_K,
\end{equation}
    where $\lambda_{\min}(M)$ is the smallest eigenvalue of the matrix $M=\nabla^2W+(1-\varepsilon)\frac{\omega'(|x|)}{|x|}\text{Id}$, $\|\nabla^2u\|^2_2$ is the Hilbert-Schmidt norm of the Hessian of $u$.
\end{lem}
\begin{rem}
  In fact, the inequality \eqref{improve} is also applicable to dual quermassintegrals, which can be regarded as a special case of the measure in Lemma \ref{improve of Cordero Rotem} when $\phi\equiv 1$. Comparing with the inequality 
$$\int_{K}\|\nabla^2u\|^2_2+\langle\nabla^2W\nabla u,\nabla u\rangle d\mu_K\geq \frac{1}{n}\mu(K),$$
   established by Cordero-Erausquin and Rotem \cite{CR23}, the inequality \eqref{improve} further modified the exponent $\frac{1}{n}$ to $\frac{1}{q}$  by taking $\varepsilon=0$, which can be used to improve the exponent of the $L_p$-BMI for $\tilde{V}_{q,\phi}$.
\end{rem}
\begin{proof}
    Let $W(x)=\omega(|x|)$, $r(x):=\frac{m|x|^2}{n}$, $m$ is a positive constant to be determined later. Since $u$ is even, then $u-r$ is also even, and $\partial_i(u-r)$ is odd. Since $\omega(t)=(n-q)\log t-\log\phi(t)$ and $\phi$ is non-increasing, then $\omega(t)$ is non-decreasing when $q\leq n$, then by the weighted Poincar\'e inequality (see \cite{CR23}), 
    \begin{align*}
        \int_{K}\|\nabla^2(u-r)\|_2^2d\mu_K&\geq\int_{K}\frac{\omega'(|x|)}{|x|}|\nabla(u-r)|^2d\mu_K\\
    &=\int_{K}\frac{\omega'(|x|)}{|x|}\left(|\nabla u|^2-2\langle\nabla u,\nabla r\rangle+|\nabla r|^2\right)d\mu_K\\
        \nonumber &=\int_{K}\frac{\omega'(|x|)}{|x|}\left(|\nabla u|^2-\frac{4m}{n}\langle\nabla u,x\rangle +\frac{4m^2|x|^2}{n^2}\right)d\mu_K\\
    &=\int_{K}\frac{\omega'(|x|)}{|x|}|\nabla u|^2-\frac{4m}{n}\langle\nabla u,\nabla W\rangle +\frac{4m^2}{n^2}\omega'(|x|)|x|d\mu_K.
    \end{align*}
    For an $n\times n$ matrix $A$, denote  the traceless part of $A$ by $\hat{A}=A-\frac{\text{tr(A)}}{n}\text{Id}$, then $\hat{\nabla}^2r=0, \|\hat{\nabla}^2u\|^2_2=\|\nabla^2u\|^2_2-\frac{(\Delta u)^2}{n}$. Since $u$ is a solution of \eqref{weight laplace equation}, then
    \begin{equation}\label{traceless}
    \begin{split}
        \|\nabla^2(u-r)\|^2_2&=\|\hat{\nabla}^2u\|^2_2+\frac{\left(\Delta (u-r)\right)^2}{n}\\
        &=\|\nabla^2u\|^2_2-\frac{(\Delta u)^2}{n}+\frac{(\Delta u-2m)^2}{n}\\
        &=\|\nabla^2u\|^2_2-\frac{4m}{n}(\langle\nabla W,\nabla u\rangle+1)+\frac{4m^2}{n}\\
        &=\|\nabla^2u\|^2_2-\frac{4m}{n}\langle\nabla W,\nabla u\rangle-\frac{4m}{n}+\frac{4m^2}{n}.
        \end{split}
    \end{equation}
    Hence by the fact that $\phi$ is non-increasing and taking $m=\frac{n}{2q}$ in \eqref{traceless}, 
    \begin{align*}
        &\int_{K}(1-\varepsilon)\|\nabla^2u\|^2_2+\langle\nabla^2W\nabla u,\nabla u\rangle d\mu_K\notag\\
        &=\int_{K}(1-\varepsilon)\left(\|\nabla^2(u-r)\|^2_2+\frac{4m}{n}\langle\nabla W,\nabla u\rangle+\frac{4m}{n}-\frac{4m^2}{n}\right)+\langle\nabla^2W\nabla u,\nabla u\rangle d\mu_K\notag\\
        &\geq\int_{K}(1-\varepsilon)\left(\frac{\omega'(|x|)}{|x|}|\nabla u|^2+\frac{4m^2}{n^2}\omega'(|x|)|x|+\frac{4m}{n}-\frac{4m^2}{n}\right)+\langle\nabla^2W\nabla u,\nabla u\rangle d\mu_K\notag\\
&=\int_{K}\left\langle\left(\nabla^2W+(1-\varepsilon)\frac{\omega'(|x|)}{|x|}\text{Id}\right)\nabla u,\nabla u\right\rangle d\mu_K\notag\\
&+\int_{K}(1-\varepsilon)\left[\frac{4m^2}{n^2}\left(n-q-\frac{\phi'(|x|)|x|}{\phi(|x|)}\right)+\frac{4m}{n}-\frac{4m^2}{n}\right]d\mu_K\notag\\
&\geq \int_{K}\left\langle\left(\nabla^2W+(1-\varepsilon)\frac{\omega'(|x|)}{|x|}\text{Id}\right)\nabla u,\nabla u\right\rangle+\frac{1-\varepsilon}{q}d\mu_K\notag\\
&\geq \int_{K}\lambda_{\min}(M)|\nabla u|^2+\frac{1-\varepsilon}{q}d\mu_K,
\end{align*} 
where $M=\nabla^2W+(1-\varepsilon)\frac{\omega'(|x|)}{|x|}\text{Id}$ and $\lambda_{\min}(M)$ is the smallest eigenvalue of the matrix $M$.
   
\end{proof}

\section{Proof of Theorem \ref{thm8}}\label{pf1}
For the $L_p$-BMI of homogeneous functionals (e.g. volume), the case $p\geq1$
 follows immediately from a well-known classical homogeneity argument once the inequality for $p=1$
 is established. This argument, however, breaks down for non-homogeneous functionals.
Hosle-Kolesnikov-Livshyts \cite[Proposition 1.4]{HKL21} proved that the aforementioned property still holds for radially symmetric log-concave measures, even without any homogeneity assumption. Furthermore, they showed that if the $L_p$-BMI holds for some $p>0$, then it also remains valid for any $p'>p$.

For the measure $\tilde{V}_{q,\phi}$ studied in this paper, we cannot directly apply Proposition 1.4 in \cite{HKL21}, essentially because the exponent shifts from $1/n$ to $1/q$. Nevertheless, through an almost line-by-line argument, we derive an analogous result to that in \cite{HKL21}, as presented below.
\begin{lem}\label{lp for p>1}
    Fix $p,q>0$. If $\mu:=\tilde{V}_{q,\phi}$ satisfies the $L_p$-BMI \eqref{bmi} for $s=\frac{p}{q}$ and $\phi$ is non-increasing, then for any $p'\geq p$, \eqref{bmi} still holds for $s=\frac{p'}{q}$. In particular, if 
    \begin{equation*}
        \tilde{V}_{q,\phi}((1-\lambda)K+\lambda L)^{\frac{1}{q}}\geq(1-\lambda)\tilde{V}_{q,\phi}(K)^{\frac{1}{q}}+\lambda\tilde{V}_{q,\phi}(L)^{\frac{1}{q}}, \quad \forall~\lambda\in [0,1], 
    \end{equation*}
    then for $p\geq1$,
    \begin{equation*}
        \tilde{V}_{q,\phi}((1-\lambda)K+_p\lambda L)^{\frac{p}{q}}\geq(1-\lambda)\tilde{V}_{q,\phi}(K)^{\frac{p}{q}}+\lambda\tilde{V}_{q,\phi}(L)^{\frac{p}{q}}, \quad \forall~\lambda\in [0,1]. 
    \end{equation*}
\end{lem}
\begin{proof}
    By the local-to-global argument in section \ref{section2}, it is only needed to discuss the local $L_p$-BMI \eqref{equivalent after Newmann problem} when $s=\frac{p}{q}$. We want to deduce the following inequality
    \begin{equation*}
        \int_{K}\|\nabla^2u\|_2^2+\langle\nabla^2W\nabla u,\nabla u\rangle d\mu_K\geq(1-p')\int_{\partial K}\frac{u_{\nu}^2}{h_K}d\mu_{\partial K}+\frac{p'}{q}\mu(K)
    \end{equation*}
    from $p'=p$ to $p'\geq p$. By a simple argument, it is sufficient to prove
    \begin{equation}
        \label{p+ inequality}
        \mu(K)\int_{\partial K}\frac{\psi^2}{\langle x,\nu\rangle}d\mu_{\partial K}\geq\frac{1}{q}\left(\int_{\partial K}\psi d\mu_{\partial K}\right)^2.
    \end{equation}
    By the calculation of \cite{N,Li1,Li2,Li3}, write
    \begin{align*}
        \mu(K)&=\int_0^1\int_{\partial K}\langle x,\nu\rangle t^{n-1}\cdot\left(|tx|^{q-n}\phi(t|x|)\right)d\mathcal{H}^{n-1}(x)dt\\
        &=\int_0^1\int_{\partial K}\langle x,\nu\rangle t^{q-1}\cdot\left(|x|^{q-n}\phi(t|x|)\right)d\mathcal{H}^{n-1}(x)dt\\
        &\geq\int_0^1\int_{\partial K}\langle x,\nu\rangle t^{q-1}\cdot\left(|x|^{q-n}\phi(|x|)\right)d\mathcal{H}^{n-1}(x)dt\\
        &=\int_{0}^{1}t^{q-1}dt\int_{\partial K}\langle x,\nu\rangle d\mu_{\partial K}=\frac{1}{q}\int_{\partial K}\langle x,\nu \rangle d\mu_{\partial K},
    \end{align*}
    where the inequality is because $\phi$ is non-increasing. Using the Cauchy-Schwarz inequality, 
    \begin{equation*}
        \int_{\partial K}\langle x,\nu\rangle d\mu_{\partial K}\int_{\partial K}\frac{\psi^2}{\langle x,\nu\rangle}d\mu_{\partial K}\geq\left(\int_{\partial K}\psi d\mu_{\partial K}\right)^2,
    \end{equation*}
    which implies \eqref{p+ inequality}.
\end{proof}

Based on the preparation in section \ref{section2} and lemma \ref{lp for p>1}, the key of proving Theorem \ref{thm8} is to obtain the inequality \eqref{equivalent after Newmann problem} for the case $p=1$. Since the right hand side of \eqref{equivalent after Newmann problem}
$$(1-p)\int_{\partial K}\frac{u_{\nu}^2}{h_K}d\mu_{\partial K}+s\mu(K)= s\mu(K)$$
when $p=1$. Next we mainly focus on dealing with the left hand side of \eqref{equivalent after Newmann problem}. 
According to $\phi\in\mathcal{N}^0_n$, we derive $\phi(t)$ is non-increasing and $t\mapsto\log\phi(e^t)$ is concave, then $(\log\phi(t))''+\frac{(\log\phi(t))'}{t}\leq 0$, which implies $\omega(t)$ is non-decreasing and $\omega''(t)+\frac{\omega'(t)}{t}\geq 0$. Then by taking $\varepsilon=0$ in Lemma \ref{improve of Cordero Rotem},
\begin{align*}
&\int_{K}\|\nabla^2u\|^2_2+\langle\nabla^2W\nabla u,\nabla u\rangle d\mu_K\notag\\
&\geq \int_{K}\left\langle\left(\nabla^2W+\frac{\omega'(|x|)}{|x|}\text{Id}\right)\nabla u,\nabla u\right\rangle+\frac{1}{q}d\mu_K\notag\\
&\geq \frac{1}{q}\mu(K),
\end{align*}
where the last inequality comes from
 \begin{align*}
\nabla^2W+\frac{\omega'(|x|)}{|x|}\text{Id}&=\omega''(|x|)\frac{x}{|x|}\otimes\frac{x}{|x|}+\frac{\omega'(|x|)}{|x|}\left(\text{Id}-\frac{x}{|x|}\otimes\frac{x}{|x|}\right)+\frac{\omega'(|x|)}{|x|}\text{Id}\\
&=\left(\omega''(|x|)+\frac{\omega'(|x|)}{|x|}\right)\frac{x}{|x|}\otimes\frac{x}{|x|}
+\frac{2\omega'(|x|^2)}{|x|^2}\left(\text{Id}-\frac{x}{|x|}\otimes\frac{x}{|x|}\right)\\
&\geq\left(\omega''(|x|)+\frac{\omega'(|x|)}{|x|}\right)\frac{x}{|x|}\otimes\frac{x}{|x|}\geq0.
\end{align*}  
Hence, by choosing $s=\frac{1}{q}$, it is easy to see that \eqref{equivalent after Newmann problem} holds for the case $p= 1, q\in(0,n]$. Combining lemma \ref{lp for p>1}, Theorem \ref{thm8} is proved. 

\section{Proof of Theorem \ref{thm9}}\label{pf2}

In this section, we mainly consider the case $p<1$, which will be more complicated to prove \eqref{equivalent after Newmann problem}, and the key is dealing with the term $\int_{\partial K}\frac{u_{\nu}^2}{h_K}d\mu_{\partial K}$. 

By John's lemma, we get $rB^n_2\subset K\subset \sqrt{n}rB^n_2$ for some $r>0$. Then combining with the definition of $\omega(t)$, divergence theorem and Cauchy-Schwarz inequality, we get for any $\xi>0$,
\begin{equation}\label{div}
\begin{split}
  \int_{\partial K}\frac{u_{\nu}^2}{h_K}d\mu_{\partial K} &\leq \frac{1}{r^2}\int_{\partial K}|\nabla u|^2\langle x,\nu \rangle d\mu_{\partial K}\\
  &= \frac{1}{r^2}\int_K div(|\nabla u|^2e^{-W}x)dx\\
  &\leq \frac{1}{r^2}\int_K\left(\xi\|\nabla^2u\|^2_2+\left(n+\frac{|x|^2}{\xi}-\langle x, \nabla W\rangle\right)|\nabla u|^2\right)d\mu_K\\
  &= \frac{1}{r^2}\int_K\left(\xi\|\nabla^2u\|^2_2+\left(\frac{|x|^2}{\xi}+q+\frac{\phi'(|x|)|x|}{\phi(|x|)}\right)|\nabla u|^2\right)d\mu_K.  
\end{split}
\end{equation}

Combining with \eqref{div} and Lemma \ref{improve of Cordero Rotem}, \eqref{equivalent after Newmann problem} can be reduced to 
\begin{equation}\label{712}
\begin{split}
  &\int_K \left[\underbrace{\lambda_{\min}(M)-\frac{1-p}{r^2}\left(\frac{|x|^2}{\xi}+q+\frac{\phi'(|x|)|x|}{\phi(|x|)}\right)}_{:=A}\right]|\nabla u|^2d\mu_K\\
  &+\int_K\left(\underbrace{\varepsilon-\frac{(1-p)\xi}{r^2}}_{:=B}\right)\|\nabla^2u\|^2_2d\mu_K+\int_K\left(\underbrace{\frac{1-\varepsilon}{q}-s}_{:=C}\right)d\mu_K\geq 0.
\end{split}
\end{equation}
Hence we only need to prove
$$A\geq 0,~B\geq 0~\mbox{and}~C\geq 0.$$
According to $rB^n_2\subset K\subset \sqrt{n}rB^n_2$, $|x|^2\leq nr^2$ for any $x\in K$. Thus we derive $B\geq 0,C\geq 0$ by choosing 
\begin{equation}\label{choose}
   \varepsilon=\frac{1-p}{nr^2}|x|^2, \xi=\frac{|x|^2}{n}, s=\frac{p}{q}. 
\end{equation}
In order to prove $A\geq 0$, it is crucial to deal with the smallest eigenvalue of $M$. Since $M=\nabla^2W+(1-\varepsilon)\frac{\omega'(|x|)}{|x|}\text{Id}$, it is easy to deduce that the eigenvalues of $M$
  $$\lambda(M)=\left(\underbrace{(2-\varepsilon)\frac{\omega'(|x|)}{|x|},\cdots,(2-\varepsilon)\frac{\omega'(|x|)}{|x|}}_{n-1},\omega''(|x|)+(1-\varepsilon)\frac{\omega'(|x|)}{|x|}\right).$$
Next we divide into two cases to discuss the smallest eigenvalue $\lambda_{\min}(M)$.

\textbf{Case 1:} 
$\lambda_{\min}(M)=(2-\varepsilon)\frac{\omega'(|x|)}{|x|}$, i.e. $t(\log \phi)'-t^2(\log\phi)''\geq 2(n-q)$.

According to $rB^n_2\subset K\subset \sqrt{n}rB^n_2$ and $\phi$ is non-increasing, then 
\begin{eqnarray*}
  A&\geq&\left(2-\frac{1-p}{nr^2}|x|^2\right)\left(\frac{n-q}{|x|^2}-\frac{\phi'(|x|)}{\phi(|x|)|x|}\right)-\frac{1-p}{r^2}\left(n+q+\frac{\phi'(|x|)}{\phi(|x|)}|x|\right)\\
  &=&\frac{2(n-q)}{|x|^2}-\frac{1-p}{nr^2}(n^2+n+nq-q)-\frac{2\phi'(|x|)}{\phi(|x|)|x|}+\left(\frac{1-p}{nr^2}-\frac{1-p}{r^2}\right)\frac{\phi'(|x|)}{\phi(|x|)}|x| \\
  &\geq& \frac{1}{r^2}\left(\frac{2(n-q)}{n}-\frac{1-p}{n}(n^2+n+nq-q)\right),
\end{eqnarray*}
by taking $\varepsilon, \xi$ as in \eqref{choose}. Hence $A\geq0$ is equivalent to choosing
$$1-\frac{2(n-q)}{n^2+n+nq-q}\leq p<1.$$
In this case, \eqref{BMI1} is proved for $p\in [1-\frac{2(n-q)}{n^2+n+nq-q},1), q\in(0,n]$.

\textbf{Case 2:} 
$\lambda_{\min}(M)=\omega''(|x|)+(1-\varepsilon)\frac{\omega'(|x|)}{|x|}$, i.e. $t(\log \phi)'-t^2(\log\phi)''\leq 2(n-q)$.

Since $\phi\in\mathcal{N}^{\tau}_n$ for $\tau>0$, then $-(\log\phi(t))''-\frac{(\log\phi(t))'}{t}\geq \frac{\tau}{t^2}$, i.e. $\omega''(t)+\frac{\omega'(t)}{t}\geq \frac{\tau}{t^2}$. Thus  
$$\lambda_{\min}(M)=\omega''(|x|)+(1-\varepsilon)\frac{\omega'(|x|)}{|x|}\geq \frac{\tau}{|x|^2}-\varepsilon\frac{\omega'(|x|)}{|x|}.$$
Combining with $rB^n_2\subset K\subset \sqrt{n}rB^n_2$ and $\phi$ is non-increasing, it implies that 
\begin{eqnarray*}
  A&\geq&\frac{\tau}{|x|^2}-\frac{1-p}{nr^2}\left(n-q-\frac{\phi'(|x|)}{\phi(|x|)}|x|\right)-\frac{1-p}{r^2}\left(n+q+\frac{\phi'(|x|)}{\phi(|x|)}|x|\right)\\
  &\geq&\frac{\tau}{nr^2}-\frac{(1-p)(n-q)}{nr^2}-\frac{(1-p)(n+q)}{r^2}+\left(\frac{1-p}{nr^2}-\frac{1-p}{r^2}\right)\frac{\phi'(|x|)}{\phi(|x|)}|x| \\
  &\geq& \frac{\tau}{nr^2}-\frac{1-p}{nr^2}(n^2+n+nq-q), 
\end{eqnarray*}
by taking $\varepsilon, \xi$ as in \eqref{choose}.
Thus $A\geq0$ is equivalent to choosing
$$1-\frac{\tau}{n^2+n+nq-q}\leq p<1.$$
Hence \eqref{BMI1} is proved for $p\in [1-\frac{\tau}{n^2+n+nq-q},1), q\in(0,n]$ in this case. The proof of Theorem \ref{thm9} is completed.

\section*{Acknowledgement}
The authors would like to express sincere gratitude to professor Yong Huang for his constant guidance.

\end{document}